\documentclass[11 pt]{article}
\usepackage{color,amsthm,amsmath,array,amssymb,amscd,amsfonts,latexsym, url,wasysym}
\usepackage{graphicx}
\usepackage[colorlinks,allcolors=blue]{hyperref}
\usepackage{xcolor}
\usepackage{lmodern}
\usepackage{tikz}
\usetikzlibrary{decorations.pathreplacing}

\usepackage{ytableau}
\usepackage{mathtools}
\usepackage{listings}
\usepackage{longtable}
\usepackage{pythonhighlight}
\usepackage{cite}
\usepackage{rotating}

\newtheorem{theo}{Theorem}[section]
\newtheorem{prop}[theo]{Proposition}

\newtheorem{lemm}[theo]{Lemma}
\newtheorem{coro}[theo]{Corollary}
\newtheorem{rema}[theo]{Remark}

\newtheorem{example}[theo]{Example}

\title{\bf Newton polytope of good symmetric polynomials}
\author{Duc-Khanh Nguyen, Nguyen Thi Ngoc Giao, Dang Tuan Hiep, Do Le Hai Thuy}

\date{}

\newfont{\gothic}{eufb10}

\begin{document}
\maketitle
\begin{abstract} 
 We introduce a general class of symmetric polynomials that have saturated Newton polytope and their Newton polytope has integer decomposition property. The class covers numerous previously studied symmetric polynomials.
\end{abstract}

\section{Introduction}
In combinatorics, if a convex polytope equals the convex hull of its integer points, we say that it is a lattice polytope. Studying lattice polytopes is important because of their connections in many other domains. For instance, in mathematical optimization, if a system of linear inequalities defines a polytope, then we can use linear programming to solve integer programming problems for this system (see \cite{barvinok2017lattice}). In algebraic geometry, lattice polytopes are used to study projective toric varieties (see \cite{cox2011toric, fulton2016introduction}). The Newton polytope is a lattice polytope associated with a polynomial: it is the convex hull of exponent vectors. The Newton polytope is a central object in tropical geometry (see \cite{kazarnovskii2021newton}), and they are used to characterizing Grobner bases (see \cite{sturmfels1996grobner}).\\

Lattice polytopes are studied by Ehrhart polynomials (see \cite{eugene1962polyedres}). Important properties of Ehrhart polynomials such as unimodality and log-concavity are related to the integer decomposition property (IDP) of the lattice polytope (see \cite{ohsugi2006special, bruns2007h, schepers2013unimodality}). In \cite{bayer2020lattice}, the authors studied the Newton polytope of inflated symmetric Grothendieck polynomials. The saturated property (SNP) of inflated symmetric Grothendieck polynomials in \cite{bayer2020lattice} generalizes the SNP of symmetric Grothendieck polynomials in \cite{escobar2017newton}. The SNP of the inflated symmetric Grothendieck polynomials is an important point to derive the IDP of their Newton polytope.\\

In this paper, we introduce a general class of symmetric polynomials that has SNP with Newton polytope has IDP (see Theorem \ref{GoodSNPIDP} and Corollary \ref{goodsign}). Our class covers symmetric polynomials in \cite{escobar2017newton, monical2019newton, bayer2020lattice, matherne2022newton}: symmetric Grothendieck polynomials, inflated symmetric Grothendieck polynomials, Stembridge's symmetric polynomials associated with totally nonnegative matrices, cycle index polynomials, Reutenauer's symmetric polynomials, Schur $P$-polynomials and Schur $Q$-polynomials, Stanley's symmetric polynomials, chromatic symmetric polynomials of co-bipartite graphs, indifference graphs of Dyck paths, incomparability graphs of (3+1)-free posets. It also covers other symmetric polynomials, for instance, dual Grothendieck polynomials in \cite{lam2007combinatorial}.\\

\textbf{Acknowledgments:} This work is supported by the Ministry of Education and Training, Vietnam, under project code B2022-CTT-02: ``Study some combinatorial models in Representation Theory", 2022-2023 (Decision No.1323/QD-BGDDT, May 19, 2022). Khanh was partially supported by the NSF grant DMS-1855592 of Prof. Cristian Lenart. Hiep would like to thank Vietnam Institute for advanced study in Mathematics for the very kind support and hospitality during his visit. We are grateful to the referee for valuable comments to improve the text.

\section{Newton polytope}\label{Newtonpolytopes}

A \textbf{polytope} $\mathcal{P}$ in $\mathbb{R}^m$ is the convex hull $Conv(v_1,\dots,v_k)$ of finite many points $v_1,\dots,v_k \in \mathbb{R}^m$. The \textbf{vertex set} of $\mathcal{P}$ is the minimal set $V$ in $\mathbb{R}^m$ such that $\mathcal{P} = Conv(V)$. Algebraically, a point $v \in \mathcal{P}$ is a \textbf{vertex} if, $v=tw+(1-t)u$ for some $w,u \in \mathcal{P}$, $t \in (0,1)$ implies $w=u=v$. We say that $\mathcal{P}$ is a \textbf{lattice polytope} if $V$ is a subset of $\mathbb{Z}^m$.

\begin{example}\label{P} The convex hull $\mathcal{P}$ of twelve points in $\mathbb{R}^3$ below is a lattice polytope.$$ (3,1,0), (3,0,1), (1,0,3), (0,1,3), (0,3,1), (1,3,0),$$ 
$$(2,2,0), (2,0,2), (0,2,2),$$
$$(2,1,1), (1,1,2), (1,2,1).$$
The permutations of $(3,1,0)$ are vertices of the polytope $\mathcal{P}$. In the picture below, $\mathcal{P}$ is the blue hexagon. 
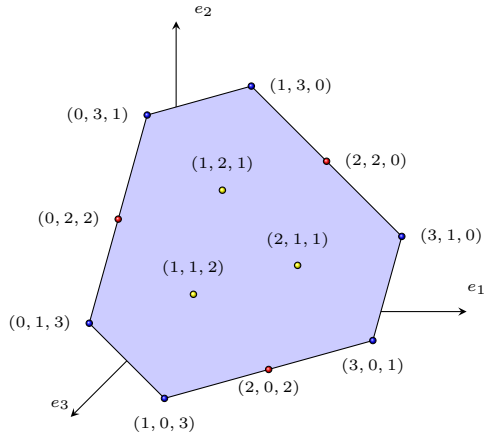
\begin{center}
\begin{tikzpicture}[scale=1]
	\def\a{blue}
	\def\b{red}
	\def\c{pink}
	\def\d{violet}
	\def\e{orange}
	\def\opacity{100}
	\tikzstyle{point1}=[ball color=blue, circle, draw=black, inner sep=0.03cm]
	\tikzstyle{point2}=[ball color=red, circle, draw=black, inner sep=0.03cm]
	\tikzstyle{point3}=[ball color=yellow, circle, draw=black, inner sep=0.03cm]
	
	\node (O0) at (0,0,0)[]{};
	\node (O1) at (4,0,0)[]{};
	\node (O2) at (0,4,0)[]{};
	\node (O3) at (0,0,4)[]{};
	
	\node at (O0) [left = 1mm]{\tiny$O$}; 
	\node at (O1) [above = 1mm]{\tiny$e_1$}; 
	\node at (O2) [right = 1mm]{\tiny$e_2$};
	\node at (O3) [above = 1mm]{\tiny$e_3$};
	 
	\draw[-{stealth[scale=3.0]}] (O0) -- (O1);
	\draw[-{stealth[scale=3.0]}] (O0) -- (O2);
	\draw[-{stealth[scale=3.0]}] (O0) -- (O3);
	
	\filldraw[fill=\a!20,rounded corners=0.5pt] (3,1,0) -- (3,0,1) -- (1,0,3) -- (0,1,3) -- (0,3,1) -- (1,3,0) -- cycle;

	\node (A1) at (3,1,0)[point1]{};
	\node (A2) at (3,0,1)[point1]{};
	\node (A3) at (1,0,3)[point1]{};
	\node (A4) at (0,1,3)[point1]{};
	\node (A5) at (0,3,1)[point1]{};
	\node (A6) at (1,3,0)[point1]{};
	
	\node (B1) at (2,2,0)[point2]{};
	\node (B2) at (2,0,2)[point2]{};
	\node (B3) at (0,2,2)[point2]{};
	
	\node (C1) at (2,1,1)[point3]{};
	\node (C2) at (1,2,1)[point3]{};
	\node (C3) at (1,1,2)[point3]{};
	
	\node at (A1) [right = 1mm]{\tiny$(3,1,0)$}; 
	\node at (A2) [below = 1mm]{\tiny$(3,0,1)$};
	\node at (A3) [below = 1mm]{\tiny$(1,0,3)$};
	\node at (A4) [left = 1mm]{\tiny$(0,1,3)$};
	\node at (A5) [left = 1mm]{\tiny$(0,3,1)$};
	\node at (A6) [right = 1mm]{\tiny$(1,3,0)$};
	
	\node at (B1) [right = 1mm]{\tiny$(2,2,0)$}; 
	\node at (B2) [below]{\tiny$(2,0,2)$};
	\node at (B3) [left = 1mm]{\tiny$(0,2,2)$};
	
	\node at (C1) [above = 1mm]{\tiny$(2,1,1)$}; 
	\node at (C2) [above = 1mm]{\tiny$(1,2,1)$};
	\node at (C3) [above = 1mm]{\tiny$(1,1,2)$};
	
\end{tikzpicture}
\end{center}
\end{example}

Let $\mathcal{P}$ be a lattice polytope. For a positive integer $t$, let $t\mathcal{P}=\{t v \mid v \in \mathcal{P}\}$. We say that $\mathcal{P}$ has \textbf{integer decomposition property (IDP)} if, for any positive integer $t$ and $p \in t\mathcal{P}\cap \mathbb{Z}^m$, there are $t$ points $v_1, \dots, v_t \in \mathcal{P} \cap \mathbb{Z}^m$ such that $p=v_1+\dots+v_t$. 

\begin{example}\label{tP} Let $\mathcal{P}$ be the lattice polytope in Example \ref{P}. It is known that $\mathcal{P}$ has IDP (\cite[Proposition 11]{bayer2020lattice}).  For instance, $3\mathcal{P}$ is the convex hull of six points \begin{equation*}
    (9,3,0), (9,0,3), (3,0,9), (0,3,9), (0,9,3), (3,9,0).
\end{equation*} 
We see that $(9,2,1) \in 3\mathcal{P} \cap \mathbb{Z}^3$ and is the sum of three points in $\mathcal{P} \cap \mathbb{Z}^3$.
$$(9,2,1)=(3,1,0) + (3,1,0) + (3,0,1).$$
\end{example} 

\begin{example} Let $\mathcal{G}$ be convex hull of four points 
\begin{equation*}
    (0,0,0), (1,0,0), (0,0,1), (1,2,1).
\end{equation*}
The elements in $\mathcal{G}\cap \mathbb{Z}^3$ are
\begin{equation*}
    (0,0,0), (1,0,0), (0,0,1), (1,2,1).
\end{equation*}
We have $(1,1,1) \in 2\mathcal{G}\cap \mathbb{Z}^3$, but it can not be written as a sum of two points in $\mathcal{G}\cap \mathbb{Z}^3$. So $\mathcal{G}$ does not have IDP.
\end{example}

Let $f(x) = \sum\limits_{\alpha \in \mathbb{Z}_{\geq 0}^m}c_\alpha x^\alpha \in \mathbb{C}[x_1,\dots,x_m]$. The \textbf{support} of $f$ is defined by $$ Supp(f)=\{\alpha \in \mathbb{Z}_{\geq 0}^m \mid c_\alpha \ne 0\}.$$ The \textbf{Newton polytope} of $f$ is defined by $$Newton(f) = Conv(Supp(f)).$$ We say that $f$ has \textbf{satured Newton polytope (SNP)} if $Newton(f) \cap \mathbb{Z}^m = Supp(f)$. 

\begin{example} \label{Newton(f)}
Let $f(x_1,x_2,x_3)$ be the polynomial 

\begin{equation*} 
\begin{split}
    & x^{(3,1,0)} + x^{(3,0,1)} + x^{(1,0,3)} + x^{(0,1,3)} + x^{(0,3,1)} + x^{(1,3,0)}\\
    + & x^{(2,2,0)} + x^{(2,0,2)} + x^{(0,2,2)}\\
    + & 2x^{(2,1,1)} + 2x^{(1,1,2)} + 2x^{(1,2,1)}.
\end{split}
\end{equation*}

The set $Supp(f)$ contains twelve points in Example \ref{P}. Then $Newton(f)$ is the polytope $\mathcal{P}$ in Example \ref{P}. Since $Newton(f) \cap \mathbb{Z}^3 = Supp(f)$, $f$ has SNP. 
\end{example}

\section{Schur polynomials}\label{Schurpolynomials}
A \textbf{partition} with at most $m$ parts is a sequence of weakly decreasing nonnegative integers $\lambda=(\lambda_1, \dots, \lambda_m)$. The \textbf{size} of partition $\lambda$ is defined by $|\lambda|= \sum\limits_{i=1}^m \lambda_i$. Each partition $\lambda$ is presented by a \textbf{Young diagram $Y(\lambda)$} that is a collection of boxes such that the leftmost boxes of each row are in a column, and the numbers of boxes from the top row to bottom row are $\lambda_1,\lambda_2,\dots$, respectively. A \textbf{semistandard Young tableau} of shape $\lambda$ with entries from $\{1,\dots,m\}$ is a filling of the Young diagram $Y(\lambda)$ by the ordered alphabet $\{1<\dots<m\}$ such that the entries in each column are strictly increasing from top to bottom, and the entries in each row are weakly increasing from left to right. A Young tableau $T$ is said to have \textbf{content} $\alpha=(\alpha_1,\alpha_2,\dots)$ if $\alpha_i$ is the number of entries $i$ in the tableau $T$. We write \begin{equation*}
x^T =x^\alpha = x_1^{\alpha_1}x_2^{\alpha_2}\dots.
\end{equation*}
For each partition $\lambda$ with at most $m$ parts, the \textbf{Schur polynomial} $s_\lambda(x_1,\dots,x_m)$ is defined as the sum of $x^T$, where $T$ runs over the semistandard Young tableaux of shape $\lambda$ with filling from $\{1,\dots,m\}$.

\begin{example}\label{Schur}
Vector $(3,1,0)$ is a partition. The Young diagram of $(3,1,0)$ is 
$$
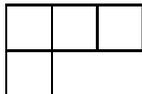
\begin{ytableau}
\,&\,&\,\\
\,\\
\end{ytableau}$$
The following filling is a semistandard tableau of shape $(3,1,0)$ and content $(1,2,1)$.
$$
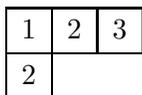
\begin{ytableau}
1&2&3\\
2
\end{ytableau} 
$$
Schur polynomial $s_{(3,1,0)}(x_1,x_2,x_3)$ is the polynomial $f$ in Example \ref{Newton(f)}.
\end{example}

\section{Good symmetric polynomials}\label{goodsymmetricpolynomials}
Let $\alpha$ and $\beta$ be partitions with at most $m$ parts. We say $\beta$ is \textbf{bigger} than $\alpha$ and write $\beta \geq \alpha$ if and only if $\beta_i \geq \alpha_i$ for all $i$. If $\alpha, \beta$ are partitions of the same size, we say $\beta$ \textbf{dominates} $\alpha$ and write $\beta \trianglerighteq \alpha$ if $\sum\limits_{i=1}^j \beta_i \geq \sum\limits_{i=1}^j \alpha_i$ for all $j \geq 1$. 

\begin{example}
$(3,1,0) < (3,3,3)$ and  $(3,2,0) \trianglerighteq (3,1,1).$
\end{example}

Let $F(x_1,\dots,x_m)$ be a linear combination of Schur polynomials associated to partitions with at most $m$ parts. We can collect Schur polynomials appearing in $F$ associated with partitions of the same size to a bracket. We say that $F$ is \textbf{good} if it satisfies the following conditions:
\begin{itemize}
    \item[(a)] The support of each bracket equals the union of supports of its Schur elements.
    \item[(b)] Suppose that there are $l+1$ brackets in condition (a). In each bracket, there is a unique $\trianglerighteq$-maximum partition. These $\trianglerighteq$-maximum partitions have a form 
    \begin{equation}\label{chain}
        \alpha=\lambda^0 < \dots < \lambda^l = \beta,
    \end{equation}
    where $\alpha \leq \beta$ are fixed partitions and for each $i>0$, $\lambda^i$ is obtained from $\lambda^{i-1}$ by adding a box in the northmost row of $\lambda^{i-1}$ such that the addition gives a Young diagram, $\alpha < \lambda^{i} \leq \beta$.
\end{itemize}

\begin{theo}\label{GoodSNPIDP}
Let $F$ be a good linear combination of Schur polynomials. Then $F$ has SNP and Newton(F) has IDP.
\end{theo}

\begin{coro}\label{goodsign}
Let $F$ be a linear combination of Schur polynomials such that the condition (a) is replaced by (a') or the condition (b) is replaced by (b') below:
\begin{itemize}
    \item[(a')] any two Schur polynomials in the same bracket of $F$ have the same sign,
    \item[(b')] there exists partitions $\bar{\lambda}, \hat{\lambda}$ so that $s_\mu$ appears in $F$ if and only if $\bar{\lambda} \leq \mu \leq \hat{\lambda}$. 
\end{itemize}
Then $F$ is a good polynomial. In particular, $F$ has SNP and $Newton(F)$ has IDP.
\end{coro}
\begin{proof}
The condition (a'), (b') are particular cases of condition (a), (b), respectively. Moreover, the partitions $\alpha, \beta$ in (b') are $\bar{\lambda}, \hat{\lambda}$, respectively. 
\end{proof}

\begin{example}\label{s310i}
Let $F(x_1,x_2,x_3)$ be 
$$
s_{(3,1,0)} - (3s_{(3,2,0)}+6s_{(3,1,1)}) + (3s_{3,3,0} + 18s_{(3,2,1)}) - (18 s_{(3,3,1)} + 4 s_{(3,2,2)}) + 44 s_{(3,3,2)} - 55 s_{(3,3,3)}. 
$$
Schur polynomials in the same bracket have the same sign. The $\trianglerighteq$-maximum partitions $\lambda^i$ for $i=0,\dots,5$ chosen from brackets have form 
$$\alpha=(3,1,0) < (3,2,0) < (3,3,0) < (3,3,1) < (3,3,2) < (3,3,3) = \beta.
$$
Hence, $F$ is a good symmetric polynomial. $Newton(F)$ is the convex hull of six different color polygons in the picture below. Each polygon is the Newton polytope of each bracket. In fact, $F$ is the inflated symmetric Grothendieck polynomial $G_{2,(3,1,0)}$ in \cite{bayer2020lattice}. Hence, $F$ has SNP and $Newton(F)$ has IDP by \cite[Proposition 21, Theorem 27]{bayer2020lattice}. 

\begin{center}
\begin{tikzpicture}[scale=1]
	\def\a{blue}
	\def\b{red}
	\def\c{pink}
	\def\d{violet}
	\def\e{orange}
	\def\opacity{100}
	\tikzstyle{point1}=[ball color=blue, circle, draw=black, inner sep=0.03cm]
	\tikzstyle{point2}=[ball color=red, circle, draw=black, inner sep=0.03cm]
	\tikzstyle{point3}=[ball color=yellow, circle, draw=black, inner sep=0.03cm]
	
	\node (O0) at (0,0,0)[]{};
	\node (O1) at (4,0,0)[]{};
	\node (O2) at (0,4,0)[]{};
	\node (O3) at (0,0,4)[]{};
	
	\node at (O0) [left = 1mm]{\tiny$O$}; 
	\node at (O1) [above = 1mm]{\tiny$e_1$}; 
	\node at (O2) [right = 1mm]{\tiny$e_2$};
	\node at (O3) [above = 1mm]{\tiny$e_3$};
	 
	\draw[-{stealth[scale=3.0]}] (O0) -- (O1);
	\draw[-{stealth[scale=3.0]}] (O0) -- (O2);
	\draw[-{stealth[scale=3.0]}] (O0) -- (O3);

	\filldraw[fill=\a!20,rounded corners=0.5pt] (3,1,0) -- (3,0,1) -- (1,0,3) -- (0,1,3) -- (0,3,1) -- (1,3,0) -- cycle;

	\node (A1) at (3,1,0)[point1]{};
	\node (A2) at (3,0,1)[point1]{};
	\node (A3) at (1,0,3)[point1]{};
	\node (A4) at (0,1,3)[point1]{};
	\node (A5) at (0,3,1)[point1]{};
	\node (A6) at (1,3,0)[point1]{};
	
	\node at (A1) [right = 1mm]{\tiny$(3,1,0)$}; 
	\node at (A2) [right = 1mm]{\tiny$(3,0,1)$};
	\node at (A3) [below = 1mm]{\tiny$(1,0,3)$};
	\node at (A4) [left = 1mm]{\tiny$(0,1,3)$};
	\node at (A5) [left = 1mm]{\tiny$(0,3,1)$};
	\node at (A6) [above = 1mm]{\tiny$(1,3,0)$};

	\filldraw[fill=\b!20,rounded corners=0.5pt] (3,2,0) -- (3,0,2) -- (2,0,3) -- (0,2,3) -- (0,3,2) -- (2,3,0) -- cycle;

	\node (A12) at (3,2,0)[point1]{};
	\node (A22) at (3,0,2)[point1]{};
	\node (A32) at (2,0,3)[point1]{};
	\node (A42) at (0,2,3)[point1]{};
	\node (A52) at (0,3,2)[point1]{};
	\node (A62) at (2,3,0)[point1]{};
	
	\node at (1,1,3)[point1]{};
	\node at (1,3,1)[point1]{};
	\node at (3,1,1)[point1]{};
	
	\node at (A32) [below = 1mm]{\tiny$(2,0,3)$};
	\node at (A42) [left = 1mm]{\tiny$(0,2,3)$};
	\node at (A52) [left = 1mm]{\tiny$(0,3,2)$};
	\node at (A62) [above = 1mm]{\tiny$(2,3,0)$};
	\node at (A22) [right = 1mm]{\tiny$(3,0,2)$};
	\node at (A12) [right = 1mm]{\tiny$(3,2,0)$};

	\filldraw[fill=\c!20,rounded corners=0.5pt] (3,3,0) -- (3,0,3) -- (0,3,3) -- cycle;

	\node (A13) at (3,3,0)[point1]{};
	\node (A23) at (3,0,3)[point1]{};
	\node (A33) at (0,3,3)[point1]{};
	
	\node at (1,2,3)[point1]{};
	\node at (2,1,3)[point1]{};
	\node at (1,3,2)[point1]{};
	\node at (2,3,1)[point1]{};
	\node at (3,1,2)[point1]{};
	\node at (3,2,1)[point1]{};
	
	\node at (A23) [below = 1mm]{\tiny$(3,0,3)$};
	\node at (A33) [left = 1mm]{\tiny$(0,3,3)$};
	\node at (A13) [above = 1mm]{\tiny$(3,3,0)$};

	\filldraw[fill=\d!20,rounded corners=0.5pt] (3,3,1) -- (3,1,3) -- (1,3,3) -- cycle;

	\node (A14) at (3,3,1)[point1]{};
	\node (A24) at (3,1,3)[point1]{};
	\node (A34) at (1,3,3)[point1]{};
	
	\node at (A24) [right = 1mm]{\tiny$(3,1,3)$};

	\filldraw[fill=\e!20,rounded corners=0.5pt] (3,3,2) -- (3,2,3) -- (2,3,3) -- cycle;

	\node (A15) at (3,3,2)[point1]{};
	\node (A25) at (3,2,3)[point1]{};
	\node (A35) at (2,3,3)[point1]{};
	
	\node at (2,2,3)[point1]{};
	\node at (2,3,2)[point1]{};
	\node at (3,2,2)[point1]{};
	
	\node at (A25) [right = 1mm]{\tiny$(3,2,3)$};

	\node (A16) at (3,3,3)[point1]{};
	\node at (A16) [right = 1mm]{\tiny$(3,3,3)$};
	
	\draw[rounded corners=0.5pt] (0,3,1) -- (0,3,3) -- (3,3,3) -- (3,3,0) -- (1,3,0) -- cycle;
	\draw[rounded corners=0.5pt] (0,3,3) -- (0,1,3);
	\draw[rounded corners=0.5pt] (1,0,3) -- (3,0,3) -- (3,3,3);
	\draw[rounded corners=0.5pt] (3,0,3) -- (3,0,1);
	\draw[rounded corners=0.5pt] (3,1,0) -- (3,3,0);
\end{tikzpicture}
\end{center}

\end{example}
The following examples tell us that when Theorem \ref{GoodSNPIDP} does not apply, we may not have a definite affirmation of SNP and IDP.  

\begin{example}
When the condition (a) fails, for instance: 
\begin{itemize}
    \item 
Let $F(x_1,x_2,x_3)$ be $s_{(3,1,0)}-s_{(2,2,0)}$. Then $F$ does not have $SNP$ because $(2,2,0) \not\in Supp(F)$, but $Newton(F)=Newton(s_{(3,1,0)})$ still has IDP.
\end{itemize}

When adding blocks to $\alpha$ in a wrong order in (b), for instance:
\begin{itemize}
    \item Let choose $\alpha=(3,1,0) < (3,1,1) < (3,2,1)=\beta$ and let $F(x_1,x_2,x_3)$ be $s_{(3,1,0)} + s_{(3,1,1)} + s_{(3,2,1)}$. Then $F$ has $SNP$.
    \item Let choose $\alpha = (6,4,0) < (6,4,1) < (6,4,2) < (6,4,3) < (6,5,3) < (6,6,3) =\beta$ and let $F(x_1,x_2,x_3)$ be $s_{(6,4,0)} + s_{(6,4,1)} + s_{(6,4,2)} + s_{(6,4,3)} + s_{(6,5,3)} + s_{(6,6,3)}$. Since $(6,5,2) \in Newton(F)\cap \mathbb{Z}^3 \setminus Supp(f)$, then $F$ does not has $SNP$.
\end{itemize}
We are not sure if there exists a symmetric polynomial that has SNP, but its Newton polytope does not have IDP.
\end{example}

We need the following facts to prove Theorem \ref{GoodSNPIDP}.

\begin{prop}(\cite[Proposition 2.5]{rado1952inequality})\label{Rado} Let $\alpha, \beta$ be partitions of the same size. Then, $Newton(s_{\alpha}) \subseteq Newton(s_{\beta})$ if and only if $\alpha \trianglelefteq \beta$.
\end{prop}

\begin{lemm}(\cite[Theorem 0.1]{escobar2017newton}) \label{Smorbit} Let $\alpha$ be a partition with at most $m$ parts. Then $s_\alpha$ has SNP with Newton polytope being the convex hull of the $S_m$-orbit of $\alpha$.
\end{lemm}

\begin{proof}[Proof of Theorem \ref{GoodSNPIDP}]
\underline{We first prove that $F$ has SNP.} We use the trick from \cite{escobar2017newton}.
\begin{itemize}
\item[1.]
Let $F=\sum\limits_{\mu}C_\mu s_\mu$ with $C_\mu \ne 0$. By condition (a) of $F$, we have
\begin{equation}\label{SuppFid1}
    Supp(F) = \bigcup\limits_{\mu} Supp(s_\mu).
\end{equation}
Then 
\begin{equation}\label{NewFid2}
    Newton(F) = Conv(\bigcup\limits_{\mu} Supp(s_\mu)).
\end{equation}
Let $\alpha=\lambda^0 < \lambda^1 < \dots < \lambda^l = \beta$ be the $\trianglerighteq$-maximum partitions in condition (b) of $F$. By Proposition \ref{Rado}, the right-hand side of (\ref{SuppFid1}) is
\begin{equation}\label{suppmu}
    \bigcup\limits_{\mu} Supp(s_\mu)=\bigcup\limits_{i=0}^l Supp(s_{\lambda^i}).
\end{equation}
Therefore, by (\ref{SuppFid1}), (\ref{suppmu}), 
\begin{equation}\label{suppFandsi}
    Supp(F)=\bigcup\limits_{i=0}^l Supp(s_{\lambda^i}).
\end{equation}
By Proposition \ref{Rado}, 
\begin{equation*}
    Conv(Supp(s_\mu)) = Newton(s_\mu) \subseteq Newton(s_{\lambda^i}) = Conv(Supp(s_{\lambda^i}))
\end{equation*}
for some $i$. It implies that the right-hand side of (\ref{NewFid2}) is 
\begin{equation}\label{Consuppmu}
Conv(\bigcup\limits_{\mu} Supp(s_\mu))=Conv(\bigcup\limits_{i=0}^l Newton(s_{\lambda^i})).    
\end{equation}
Hence by (\ref{NewFid2}), (\ref{Consuppmu}), we have
\begin{equation}\label{NewFid3}
    Newton(F) = Conv(\bigcup\limits_{i=0}^l Newton(s_{\lambda^i})).
\end{equation}
\item[2.]
Let $p$ be a point in $Newton(F) \cap \mathbb{Z}^m$. By (\ref{NewFid3}), $p$ has form $p=\sum\limits_{i=0}^l c_i v^i$ for some $v^i \in Newton(s_{\lambda^i})$, and some $c_i \in \mathbb{R}_{\geq 0}$, $\sum\limits_{i=1}^l c_i = 1$. We see that $v^i$ is not a partition in general. However, if we denote the sum of its coordinates by $|v^i|$, then $|v^i|=|\lambda^i|$. Then $|p|=\sum\limits_{i=0}^l c_i|\lambda^i|$ is between $|\lambda^0|$ and $|\lambda^l|$, because of (\ref{chain}). Thus $|p|=|\lambda^j|$ for some $j \in [0,l]$, because $\lambda^{i}$ is obtained from $\lambda^{i-1}$ by adding a box. Let $\overline{p}$ be $\sum\limits_{i=0}^l c_i\lambda^i$ and $p^{\downarrow}$ be the rearrangement of the components of $p$ into decreasing order. It was proven in \cite{escobar2017newton} that $p^{\downarrow} \trianglelefteq (\overline{p})^{\downarrow}$ (Claim B) and $(\overline{p})^{\downarrow} \trianglelefteq \lambda^j$ (Claim C). So $p^{\downarrow} \trianglelefteq \lambda^j$. By Lemma \ref{Smorbit}, Proposition \ref{Rado}, $p$ is a point in
\begin{equation}
    Newton(s_{p^{\downarrow}}) \cap \mathbb{Z}^m \subseteq Newton(s_{\lambda^j}) \cap \mathbb{Z}^m = Supp(s_{\lambda^j}) \subseteq Supp(F).
\end{equation}
Therefore we conclude that $F$ has SNP.
\end{itemize}

\underline{Now we show that $Newton(F)$ has IDP.} We use the trick from \cite{bayer2020lattice}.
\begin{itemize}
    \item[1.] We have proven that $F$ has SNP. Then by (\ref{suppFandsi}), Lemma \ref{Smorbit}, we have
    \begin{equation}\label{NewFZid}
        Newton(F) \cap \mathbb{Z}^m = Supp(F) = \bigcup\limits_{i=0}^l Supp(s_{\lambda^i}) = \bigcup\limits_{i=0}^l Newton(s_{\lambda^i})\cap \mathbb{Z}^m.
    \end{equation}

    \item[2.] Suppose that $\alpha=(\alpha_1,\dots,\alpha_m)$ and $\beta=(\beta_1,\dots,\beta_m)$. For $i = 1, \dots, m-1$, set $\lambda^{(i)}=(\beta_1,\dots,\beta_i,\alpha_{i+1},\dots,\alpha_m)$. Set $\lambda^{(0)}=\alpha$, $\lambda^{(m)}=\beta$. Then $\alpha=\lambda^{(0)} < \dots < \lambda^{(m)}=\beta$ is a subchain of (\ref{chain}). We have 
\begin{equation}\label{NewFid4}
    Newton(F) = Conv(\bigcup\limits_{i=0}^m Newton(s_{\lambda^{(i)}})).
\end{equation}
Indeed, $Newton(F)$ is the convex hull of its vertex set. We can get (\ref{NewFid4}) from (\ref{NewFid3}) by showing that a partition $\lambda^j$ not of form $\lambda^{(i)}$ is not a vertex of $Newton(F)$. It is trivial because 
$\lambda^j=\frac{1}{2}(\lambda^{j-1}+\lambda^{j+1})$.

\item[3.]For a positive integer $t$, we construct a chain of form (\ref{chain}) 
\begin{equation}\label{chain2}
    t\alpha= \Lambda^0 < \dots < \Lambda^L=t\beta.
\end{equation}
Set $F_t=\sum\limits_{i=0}^Ls_{\Lambda^i}$. Then $F_t$ is a good linear combination of Schur polynomials and $\Lambda^{(i)}=t\lambda^{(i)}$ for each $i=0,\dots,m$. By (\ref{NewFid4}), we have
\begin{equation}\label{NewFid5}
\begin{split}
    Newton(F_t)&= Conv(\bigcup\limits_{i=0}^m Newton(s_{\Lambda^{(i)}}))\\
    &=tConv(\bigcup\limits_{i=0}^m Newton(s_{\lambda^{(i)}}))\\
    &=tNewton(F).
\end{split}
\end{equation}

\item[4.] Let $p$ a point in $tNewton(F)\cap \mathbb{Z}^m$. By (\ref{NewFid5}), $p$ is a point in $Newton(F_t)\cap \mathbb{Z}$. Since $F_t$ has SNP, by (\ref{NewFZid}), it is a point in $Newton(s_{\Lambda^i}) \cap \mathbb{Z}$ for some $\Lambda^{i}$ in (\ref{chain2}). Hence, $p$ is the content of some semistandard tableau $T$ of shape $\Lambda^i$ with filling from $\{1,\dots,m\}$. For $j=1,\dots,t$, let $T_j$ be the semistandard tableau obtained by taking $j'$-th column of $T$ for $j'\equiv j \mod t $. Let $\theta(j)$ be the shape of tableau $T_j$. Let $v_j$ be the content of tableau $T_j$. Then $p=v_1+\dots+v_t$. We also have $\alpha \leq \theta(j) \leq \beta$. So there is a unique partition $\lambda^k$ in chain (\ref{chain}) such that $\theta(j) \trianglelefteq \lambda^k$. Then by Proposition \ref{Rado}, $v_j$ is a point in 
\begin{equation*}
    Newton(s_{\theta(j)}) \cap \mathbb{Z}^m \subseteq Newton(s_{\lambda^k}) \cap \mathbb{Z}^m.
\end{equation*}
So by (\ref{NewFZid}), $v_j$ is a point of $Newton(F) \cap \mathbb{Z}^m$. Therefore we conclude that $Newton(F)$ has IDP.
\end{itemize}
\end{proof}

\begin{example} In Example \ref{s310i}, the subchain $\lambda^{(i)}$ for $i=0,\dots,3$ in the proof of Theorem \ref{GoodSNPIDP} is 
\begin{equation*}
    \alpha=(3,1,0) = (3,1,0) < (3,3,0) < (3,3,3) = \beta.
\end{equation*}
In this case, $\lambda^{(0)}=\lambda^{(1)}$. The vertex set of $Newton(F)$ is the union of $S_3$-orbits of partitions $(3,1,0), (3,3,0), (3,3,3)$.

\begin{center}
\begin{tikzpicture}[scale=1]
	\def\a{blue}
	\def\b{red}
	\def\c{pink}
	\def\d{violet}
	\def\e{orange}
	\def\opacity{100}
	\tikzstyle{point1}=[ball color=blue, circle, draw=black, inner sep=0.03cm]
	\tikzstyle{point2}=[ball color=red, circle, draw=black, inner sep=0.03cm]
	\tikzstyle{point3}=[ball color=yellow, circle, draw=black, inner sep=0.03cm]
	
	\node (O0) at (0,0,0)[]{};
	\node (O1) at (4,0,0)[]{};
	\node (O2) at (0,4,0)[]{};
	\node (O3) at (0,0,4)[]{};
	
	\node at (O0) [left = 1mm]{\tiny$O$}; 
	\node at (O1) [above = 1mm]{\tiny$e_1$}; 
	\node at (O2) [above = 1mm]{\tiny$e_2$};
	\node at (O3) [above = 1mm]{\tiny$e_3$};
	 
	\draw[-{stealth[scale=3.0]}] (O0) -- (O1);
	\draw[-{stealth[scale=3.0]}] (O0) -- (O2);
	\draw[-{stealth[scale=3.0]}] (O0) -- (O3);

	\filldraw[fill=\a!20,rounded corners=0.5pt] (3,1,0) -- (3,0,1) -- (1,0,3) -- (0,1,3) -- (0,3,1) -- (1,3,0) -- cycle;

	\node (A1) at (3,1,0)[point1]{};
	\node (A2) at (3,0,1)[point1]{};
	\node (A3) at (1,0,3)[point1]{};
	\node (A4) at (0,1,3)[point1]{};
	\node (A5) at (0,3,1)[point1]{};
	\node (A6) at (1,3,0)[point1]{};
	
	\node at (A1) [right = 1mm]{\tiny$(3,1,0)$}; 
	\node at (A2) [right = 1mm]{\tiny$(3,0,1)$};
	\node at (A3) [below = 1mm]{\tiny$(1,0,3)$};
	\node at (A4) [left = 1mm]{\tiny$(0,1,3)$};
	\node at (A5) [left = 1mm]{\tiny$(0,3,1)$};
	\node at (A6) [above = 1mm]{\tiny$(1,3,0)$};

	\filldraw[fill=\b!20,rounded corners=0.5pt] (3,2,0) -- (3,0,2) -- (2,0,3) -- (0,2,3) -- (0,3,2) -- (2,3,0) -- cycle;

	\node (A12) at (3,2,0)[point1]{};
	\node (A22) at (3,0,2)[point1]{};
	\node (A32) at (2,0,3)[point1]{};
	\node (A42) at (0,2,3)[point1]{};
	\node (A52) at (0,3,2)[point1]{};
	\node (A62) at (2,3,0)[point1]{};
	
	\node at (1,1,3)[point1]{};
	\node at (1,3,1)[point1]{};
	\node at (3,1,1)[point1]{};

	\filldraw[fill=\c!20,rounded corners=0.5pt] (3,3,0) -- (3,0,3) -- (0,3,3) -- cycle;

	\node (A13) at (3,3,0)[point1]{};
	\node (A23) at (3,0,3)[point1]{};
	\node (A33) at (0,3,3)[point1]{};
	
	\node at (1,2,3)[point1]{};
	\node at (2,1,3)[point1]{};
	\node at (1,3,2)[point1]{};
	\node at (2,3,1)[point1]{};
	\node at (3,1,2)[point1]{};
	\node at (3,2,1)[point1]{};
	
	\node at (A23) [below = 1mm]{\tiny$(3,0,3)$};
	\node at (A33) [left = 1mm]{\tiny$(0,3,3)$};
	\node at (A13) [above = 1mm]{\tiny$(3,3,0)$};

	\filldraw[fill=\d!20,rounded corners=0.5pt] (3,3,1) -- (3,1,3) -- (1,3,3) -- cycle;

	\node (A14) at (3,3,1)[point1]{};
	\node (A24) at (3,1,3)[point1]{};
	\node (A34) at (1,3,3)[point1]{};

	\filldraw[fill=\e!20,rounded corners=0.5pt] (3,3,2) -- (3,2,3) -- (2,3,3) -- cycle;

	\node (A15) at (3,3,2)[point1]{};
	\node (A25) at (3,2,3)[point1]{};
	\node (A35) at (2,3,3)[point1]{};
	
	\node at (2,2,3)[point1]{};
	\node at (2,3,2)[point1]{};
	\node at (3,2,2)[point1]{};

	\node (A16) at (3,3,3)[point1]{};
	\node at (A16) [right = 1mm]{\tiny$(3,3,3)$};
	
	\draw[rounded corners=0.5pt] (0,3,1) -- (0,3,3) -- (3,3,3) -- (3,3,0) -- (1,3,0) -- cycle;
	\draw[rounded corners=0.5pt] (0,3,3) -- (0,1,3);
	\draw[rounded corners=0.5pt] (1,0,3) -- (3,0,3) -- (3,3,3);
	\draw[rounded corners=0.5pt] (3,0,3) -- (3,0,1);
	\draw[rounded corners=0.5pt] (3,1,0) -- (3,3,0);
\end{tikzpicture}
\end{center}

\end{example}

\section{Applications}\label{applications}
Theorem \ref{GoodSNPIDP}, Corollary \ref{goodsign} cover the following cases. Known results are:
\begin{itemize}
    \item SNP and IDP of inflated symmetric Grothendieck polynomials $G_{h,\lambda}$ (see \cite[Theorem 0.1]{escobar2017newton}, \cite[Proposition 21, Theorem 27]{bayer2020lattice}). Indeed, by definition
    \begin{equation*}
        G_{h,\lambda} = \sum\limits_{\mu}(-1)^{|\mu/\lambda|}b_{h,\lambda\mu} s_\mu,
    \end{equation*}
    where $b_{h,\lambda\mu}$ is the number of fillings satisfying certain conditions. So, all Schur elements in the same bracket with $s_\mu$ have the same sign $(-1)^{|\mu/\lambda|}$, and then the condition (a) is valid. By \cite[Lemma 18 (c)]{bayer2020lattice}, $b_{h,\lambda\mu}$ is nonzero if and only if $\lambda \leq \mu \leq \lambda^{(N)}$. Hence, by Corollary \ref{goodsign}, the condition (b) is valid with $\alpha = \lambda$ and $\beta = \lambda^{(N)}$. 
    \item SNP and IDP of the following symmetric polynomials in \cite{monical2019newton}: Stembridge's symmetric polynomials associated with totally nonnegative matrices (Theorem 2.28), cycle index polynomials (Theorem 2.30), Reutenauer's symmetric polynomials (Theorem 2.32), Schur $P$-polynomials and Schur $Q$-polynomials (Proposition 3.5), Stanley's symmetric polynomials (Theorem 5.8). They are particular cases of \cite[Prositions 2.5 (III)]{monical2019newton}. The proposition considers homogenous symmetric polynomials of degree $d$
    \begin{equation*}
        f = \sum_{|\mu| = d} c_\mu s_\mu
    \end{equation*}
    with suppose that there exists $\lambda$ so that $c_{\lambda} \ne 0$, $c_\mu \ne 0$ only if $\mu \trianglelefteq \lambda$, and $c_\mu \geq 0$ for all $\mu$. So, condition (a) is valid. The condition (b) is valid with $\alpha =\beta = \lambda$. More precisely, the Schur expansion of those polynomials have nonnegative coefficients by \cite{stembridge1991immanants}, \cite[page 396]{stanley1999enumerative}, \cite[page 12]{monical2019newton}, \cite{stembridge1989shifted}, \cite[Theorems 3.2, 4.1]{stanley1984number}, respectively. The condition (b) is valid with $\alpha=\beta$ and they can be found in the proofs of corresponding theorems in \cite{monical2019newton}.   
    
    \item SNP and IDP of the following symmetric polynomials in \cite{matherne2022newton}: chromatic symmetric polynomials of co-bipartite graphs (Proposition 3.1), indifference graphs of Dyck paths (Proposition 4.1), incomparability graphs of (3+1)-free posets (Theorem 5.7). They are also particular cases of \cite[Proposition 2.5 (III)]{monical2019newton} above. More precisely, the Schur-expansion of those polynomials have nonnegative coefficients by \cite[Corollary 3.6]{stanley1995symmetric}, \cite{stanley1993immanants}, \cite{gasharov1996incomparability}, respectively. Hence, condition (a) is valid. The condition (b) is valid with $\alpha =\beta$ and they are $\lambda(G), \lambda^{gr}(d), \lambda^{gr}(P)$, respectively. 
\end{itemize}
Unknown results are:
\begin{itemize}
    \item SNP and IDP of dual Grothendieck polynomials $g_\lambda$ in \cite{lam2007combinatorial}. Indeed, \cite[Theorem 9.8]{lam2007combinatorial} states that 
    \begin{equation*}
        g_\lambda = \sum\limits_{\mu} f_{\lambda}^{\mu} s_\mu,
    \end{equation*}
    where $f_\lambda^\mu$ is the number of semistandard tableaux of the skew shape $\lambda/\mu$ with entries of the $i$-th row lie in $[1,i-1]$. So, all nonzero coefficients $f_{\lambda}^\mu$ have same sign, and then the condition (a) is valid. Moreover, $f_{\lambda}^{\mu}$ is nonzero if and only if $(\lambda_1) \leq \mu \leq \lambda$. Hence, by Corollary \ref{goodsign}, the condition (b) is valid with $\alpha = (\lambda_1)$ and $\beta = \lambda$.
\end{itemize}

\begin{rema}
Though Theorem \ref{GoodSNPIDP} covers \cite[Theorem 27]{bayer2020lattice}, inside the proofs we do not need to choose $F_t$ as a generalization of $G_{th,t\lambda}$. The key point is to choose a set-up for $F_t$ so that it has SNP and $Newton(F_t) = t Newton(F)$ for any $t$. For this purpose, there are many choices for $F_t$, for instance  $\sum\limits_{i=0}^L s_{\Lambda^i}$, or $\sum\limits_{i=0}^L (-1)^i s_{\Lambda^i}$, or $G_{th,t\lambda}$ when $F=G_{h,\lambda}$, etc. Our first choice $F_t=\sum\limits_{i=0}^L s_{\Lambda^i}$ is the simplest.
\end{rema}

\bibliography{references}{}
\bibliographystyle{alpha}
\noindent Department of Mathematics and Statistics, University at Albany, Albany, NY 12222, USA.\\
E-mail: \href{khanh.mathematic@gmail.com}{khanh.mathematic@gmail.com} \\

\noindent Faculty of Advanced Science and Technology, University of Science and Technology - The
University of Da Nang, 54 Nguyen Luong Bang, Da Nang, Vietnam.\\
E-mail: \href{ngocgiao185@gmail.com}{ngocgiao185@gmail.com} \\

\noindent Department of Mathematics, Dalat University, 1 Phu Dong Thien Vuong, Ward 8, Dalat City, Lam Dong, Vietnam.\\
E-mail: \href{hiepdt@dlu.edu.vn}{hiepdt@dlu.edu.vn} \\

\noindent Institute of Mathematics, Vietnam Academy of Science and Technology, 18 Hoang Quoc Viet, Cau Giay, Hanoi, Vietnam.\\
E-mail: \href{cbl.dolehaithuy@gmail.com}{cbl.dolehaithuy@gmail.com}

\end{document}